\documentclass[a4paper,11pt]{amsart}

\usepackage{amssymb,amscd,amsthm,mathrsfs}
\usepackage[ansinew]{inputenc}
\usepackage[centertags]{amsmath}
\usepackage{graphicx,psfrag}

\newcommand{\la}{\lambda}
\newcommand{\La}{\Lambda}
\newcommand{\ep}{\epsilon}

\newcommand{\R}{\mbox{$\mathbb{R}$}}

\newcommand{\Li}{{\tilde{L}}}

\newcommand{\Di}{Di\!f\!f}

 \def\RR{{\mathbb R}}  \def\TT{{\mathbb T}}
   
 \def\ZZ{{\mathbb Z}}

\def\cA{\mathcal{A}}    
\def\cB{\mathcal{B}}    \def\cT{\mathcal{T}}
\def\cC{\mathcal{C}}   \def\cO{\mathcal{O}} \def\cU{\mathcal{U}}
\def\cD{\mathcal{D}}  \def\cJ{\mathcal{J}} \def\cP{\mathcal{P}} \def\cV{\mathcal{V}}
    \def\cW{\mathcal{W}}
\def\cF{\mathcal{F}}   \def\cR{\mathcal{R}}

\newtheorem*{teo*}{Theorem}

\newtheorem{thm}{Theorem}[section]
\newtheorem{conj}{Conjecture}

\newtheorem{cor}{Corollary}[section]
\newtheorem{lem}{Lemma}[section]
\newtheorem{prop}{Proposition}[section]
\newtheorem{prob}{Problem}
\newcommand{\bi}{\begin{itemize}}
\newcommand{\ei}{\end{itemize}}

\theoremstyle{definition}
\newtheorem{defi}{Definition}[section]
\theoremstyle{remark}
\newtheorem{rem}[]{Remark}[section]

\author{Mart\'{i}n Sambarino}
\address{CMAT, Facultad de Ciencias, Universidad de la Rep\'ublica del  Uruguay} \address{Igua 4225. Montevideo, Uruguay.}
\email{samba@cmat.edu.uy}

\title[Dominated Splitting]{A (short) survey on Dominated Splitting}

\begin{document}

\begin{abstract}
We present here the concept of Dominated Splitting and give an
account of some important results on its dynamics.

\end{abstract}

\maketitle
\begin{flushright}
\textit{Dedicated to Ricardo Ma\~n\'e (1948-1995)}
\end{flushright}

\section{Introduction}

Since the invention of Calculus a basic tool in order to get
information on a  smooth object is to look to its ``linear
approximation''. Let me start with two very simple examples to give
a rough idea of the purpose of the present paper.

First, as we learn during the first courses, when we
have a smooth function of the real line or the interval, if we know that at some point $x,$  the derivative $f'(x)$
has a definite sign,  then we know that the
function is strictly monotonous  on a neighborhood of
the point $x.$ Notice that any map $g$ smoothly close to $f$  will
have $g'(x)$ of the same sign. On the other hand if $f'(x)=0$ then
we cannot ensure what the behavior of $f$ around $x$ will be.
However, if we know that on a $C^1$-neighborhood of $f,$ all maps have the same behavior around a uniform neighborhood of
$x,$  then  we conclude that the
derivative of $f$ has a definite sign at $x$.

The second example I would like to mention comes from the study of
autonomous differential equations. Let $x'=f(x)$ be a differential
equation where $f:\Omega\to \RR^n$ is a smooth map on an open region
$\Omega.$ Suppose that we have an isolated equilibrium point $x_0$
and let $A=Df_{x_0}.$ We know that if all the eigenvalues of $A$
have negative real part then the equilibrium point $x_0$ is
(asymptotically) stable. Notice that if $g$ smoothly close to
$f,$ the corresponding differential equation will have an
asymptotically stable equilibrium point near $x_0.$ On the other
hand, if all eigenvalues of $A$ have non positive real part (but
there is one with zero real part) then we cannot ensure what the
behavior around $x_0$ will be. However, if we know that for any
$g$ on a $C^1$ neighborhood of  $f,$ the corresponding differential equation  has a unique equilibrium point in a
uniform neighborhood of $x_0$ and it is asymptotically stable, then
all the eigenvalues of $A$ have negative real part.

I mentioned the above examples to support the following idea (or
principle) in smooth dynamics, which is somehow the leitmotif of this
survey: \textit{a ``stable'' structure of the linear approximation
should allow one to describe the dynamics (or at least should impose
restrictions on it) and a ``stable'' dynamical phenomena implies some ``stable''
structure on the linear approximation}.

This principle has been beautifully accomplished in the so called
\textit{hyperbolic theory} started by D.V. Anosov and S. Smale in
the sixties and an endless list of contributors (Bowen, Franks,
Ma\~{n}\'e, Newhouse, Palis, Pugh,  Shub just to mention a few). Let us
explain it in a very informal way.  A smooth dynamical system
$f:M\to M$ is hyperbolic (on a compact invariant set $\La$) if
the tangent map $Df$ has a hyperbolic structure over the set,
meaning that the tangent bundle splits into two subbundles $T_\La=E^s\oplus E^u$ invariant
under $Df$ and the action of $Df$ on this subbundles has a uniform
behavior (uniformly contracting on $E^s$  and uniformly expanding on
$E^u$). A precise definition is given in the next section.
This hyperbolic structure of $Df$ has strong implications on the
dynamics of $f$. Locally, the dynamics splits into two
``directions'' and  in one direction points get exponentially close
by forward iteration and on the other one points get exponentially
close by backward iteration. These facts allow to successfully
describe the dynamics (see for instance \cite{S},
\cite{Sh}, \cite{R}, \cite{KH}). On the other hand this hyperbolic structure
cannot be destroyed by small smooth perturbation and ``stability''
can be get from here. Furthermore, (structurally) stable dynamical
systems are hyperbolic (a remarkable result by Ma\~n\'e \cite{M1}).

Weaker forms of hyperbolicity appeared in the literature (in the 70's): Partial Hyperbolicity and Dominated Splitting.
These (non hyperbolic) structures on the tangent map
$Df$ of a dynamical system $f:M\to M$  can not be destroyed by
small perturbations. The weaker one having this property is Dominated Splitting. Roughly speaking,
we say that an invariant set has Dominated Splitting if the tangent
bundle over the set splits into two invariant subbundles which are
invariant under $Df$ and although we do not know (a priori) that the
action of $Df$ on these subbundles has a uniform behavior
(contracting or expanding), we know that the action on one
\textit{dominates} the action on the other one (see Definition
\ref{dDS}). This ``domination'' prevents the structure to be
destroyed by small perturbations. This dominated splitting structure
may allow us to think that the dynamics of $f$ splits into two
directions. Nevertheless, since the action of $Df$ on the invariant
subbundles might not have a unform behavior we can not get
information on how the dynamics on these ``direction'' would be. And
so, trying to get dynamical information from the dominated splitting
seems hopeless.

Can we describe the dynamics of a set having Dominated Splitting?
After a first glance, the question is very naive since any dynamical
system $f:M\to M$ can be embedded into another one having Dominated
Splitting (multiplying by other systems having strong
contraction/expansion), though in this case the dynamics lives in low dimensional submanifold. We have to take a deeper look to see which are the
right questions to ask. We will see through the following pages
there are many results in this direction though there is still a
long way to a comprehensive understanding of it.

Dominated Splitting appeared in the literature as a \textit{tool} and during the last twenty years there had been a growing interest on it and it fits with many different subjects in dynamics.

It plays a central role in the  understanding of global dynamics from a ($C^1$) generic viewpoint. The Palis' conjecture (\cite{Pa},\cite{PaT}) motivated much of the work in this area. The state of the art regarding this conjecture is the remarkable result by S. Crovisier and E. Pujals (\cite{CP}). For a global view on generic $C^1$ dynamics I recommend the survey by C. Bonatti (\cite{B}) and the comprehensive work by Crovisier in \cite{C1}. Robust dynamical phenomena (in the $C^1$) topology, as robust transitivity for example, are well described in terms of Dominated Splitting as well and there are some surveys that the interested reader could consult (\cite{PS4}, \cite{P1}, \cite{P2} and also \cite{B} and the book \cite{BDV}).

It plays also a major role in studying $C^1$ generic diffeomorphisms preserving a volume form that started with the seminal work of R. Ma\~n\'e (\cite{M6}). As a reference to the subject see the paper by Avila and Bochi \cite{AB}. Recently, Avila-Crovisier-Wilkinson announced a landmark result in this area and the paper will be available soon \cite{ACW}. It is also present in the study of \textit{stable ergodicity} for diffeomorphisms preserving a volume form (see \cite{ACW}, \cite{T1}, \cite{Wi}).

The extension of the ergodic theory \cite{Bo} of hyperbolic systems (as the existence of SRB or physical measures)  to a wider class it has been the interest of researchers for a long time. Systems having Dominated Splitting are in the class where successful extensions has taken place (see for instance \cite{BDV} and references therein).

On the other hand, the (robust) lack of dominated splitting has wild dynamical consequences (see \cite{B}, \cite{BD1}, \cite{BD2} and \cite{BCS}). For instance, in \cite{BCS}, it is proven the Pesin Theory (\cite{Pe} and \cite{BaPe}) dramatically fail in the $C^1$ topology: there are hyperbolic measures where the stable and unstable manifolds of any point in the support of the measure are trivial. Nevertheless, much of Pesin theory on hyperbolic measures can be recovered in the $C^1$ category provided the Oseledet's splitting is a dominated splitting, as claimed in \cite{M6}. For instance in \cite{ABC} it is proven the existence of stable and unstable manifolds and a version of Katok's closing lemma \cite{K} (an early adaptation in two dimensions appeared in \cite{Ga2}); in \cite{G} the Katok's horseshoe construction (\cite{K}) is extended to the $C^1$ case; and Pesin's entropy formula has been extended as well (see  \cite{CCE}, \cite{ST} and \cite{T2}).

Likewise, Dominated Splitting has to do with flows, specially in the understanding of the Lorenz attractor and the theory of singular hyperbolic flows (see for instance \cite{MPP} and the book \cite{AP}).

A complete survey on the subject will result in a very long
paper or just in a nonsense collection of results. To avoid both I
tried to focus on some aspects and results to give the flavor of the
subject. What is this survey about? It is about the dynamical consequences we may extract of a set with Dominated Splitting just from the structure itself (without assuming other hypothesis like being far from homoclinic tangencies, heterodimensional cycles, or preserving a measure, etc). And so many important contributions (besides the ones mentioned in the above paragraphs) might not appear here. I
apologize for this in advance.

\textit{Acknowledgements:} I wish to thanks Sylvain Crovisier, Enrique Pujals and specially Rafael Potrie for reading a preliminary draft and for their comments, corrections and suggestions.

\section{Definitions and examples}

A dynamical systems will be  a diffeomorphism $f:M\to M$ where $M$ is
compact riemannian manifold without boundary. We denote by
$\Di^r(M)$ the space of $C^r$ diffeomorphisms endowed with the $C^r$
topology. We say that $f$ is $C^r$ generic if belongs to a $G_\delta$ set (in the Baire sense).

An $\ep$-chain form $x$ to $y$ is a finite sequence $x_0=x,x_1,...,x_n=y$ such that $dist(f(x_i),x_{i+1})<\ep,$ for  $i=0,...,n-1.$ The chain recurrent set $\cR(f)$ is the set of points $x\in M$ such that there exists an $\ep$-chain from $x$ to itself for any $\ep>0.$ A set $\La\subset\cR(f)$ is chain-transitive if for any $x,y\in\La$ there is an $\ep$-chain from $x$ to $y.$ A chain-recurrent class is a chain-transitive set which is maximal for the inclusion. Chain recurrent classes are compact, invariant and pairwise disjoint. The chain recurrent set is were the dynamic takes place (see \cite{Co}). Other central notions are the non-wandering set $\Omega(f),$ the limit set $L(f)$ and set of periodic points $Per(f).$ It is always true that $\overline{Per(f)}\subset L(f)\subset\Omega(f)\subset\cR(f)$ and, from the Connecting Lemma for pseudo-orbits \cite{BC}, they are $C^1$ generically equal to each other.

\begin{defi}\label{dHyp}
A compact invariant set $\La$ of a dynamical system $f:M\to M$ is
hyperbolic provided the tangent bundle over $\La$ splits into two
subbundles $T_\La M=E^s\oplus E^u$ such that
\begin{itemize}
\item $E^s$ and $E^u$ are invariant by $Df,$ that is $Df(E^s_x)=E^s_{f(x)}$ and $Df(E^u_x)=E^u_{f(x)}.$
\item There exists $C>0$ and $0<\la<1$ such that for any $x\in \La$
$$\|Df^n_{/E^s_x}\|\le C\la^n,\,\, n\ge 0\,\,\,\,\mbox{ and } \,\,\,\|Df^{-n}_{/E^u_x}\|\le C\la^n,\,\,\, n\ge 0.$$
\end{itemize}
\end{defi}

A periodic point $p$ is hyperbolic if its orbit $\cO(p)$ is a hyperbolic set. The index of $p$ is de dimension of $E^s_p.$ The homoclinic class of a hyperbolic periodic point is the closure of the set of transversal intersection between the stable and unstable manifolds of $\cO(p).$

Let's now recall the meaning of Dominated Splitting. This concept was
introduced by Ma\~{n}\'e in \cite{M2},\cite{M3}. Apparently this notion
already appeared in a letter from Ma\~n\'e to J. Palis when he was an
undergraduate student in Uruguay and before the celebrated conference in
Salvador de Bahia in 1971. In this letter Ma\~n\'e claimed he had a
proof of the stability conjecture, something he would achieve almost
twenty years later! Also, Pliss \cite{Pl1} and Liao \cite{L1}
used this notion independently without named it. And in \cite{HPS} this
notion is also hidden in the definition of \textit{eventually relatively
normal hyperbolicity}.

\begin{defi}\label{dDS}
Let $\La$ be an invariant set of $f:M\to M.$ We say that $\La$ has
Dominated Splitting provided the tangent bundle over $\La$ splits
into two subbundles $T_\La M=E\oplus F$ such that
\begin{enumerate}\renewcommand{\theenumi}{\roman{enumi}}
\item $E$ and $F$ are invariant by $Df.$
\item\label{e.cont} The subbundles $E$ and $F$ are continuous, i.e., $E_x$ and $F_x$ vary continuously with $x\in\La.$
\item There exist $C>0$ and $0<\la<1$ such that for any $x\in\La$
\begin{equation}\label{e.ds}
\|Df^n_{/E_x}\|\|Df^{-n}_{F_{f^n(x)}}\|\le C\la^n,\,\, n\ge 0.
\end{equation}
\end{enumerate}
\end{defi}

A way to understand the above definition is that \textit{any
direction not contained in the subbundle $E$ converges exponentially
fast to the direction $F$ under iteration of $Df.$}

 Condition \eqref{e.ds} is equivalent to the following:
\begin{equation}
\mbox{There exits $m>0$ such that }
\|Df^m_{/E_x}\|\|Df^{-m}_{F_{f^m(x)}}\|\le\frac{1}{2}
\end{equation}
and can be written also in the following form:
\begin{equation}
\mbox{For any $v_E\in E_x-\{0\}$ and $V_F\in F_x-\{0\}:$ }\,\,\,\,\,
\frac{\|Df^m_x v_E\|}{\|v_E\|}\le \frac{1}{2}\frac{\|Df^m_x
v_F\|}{\|v_F\|}.
\end{equation}
Also, recalling that the co-norm or mininorm of a linear map $A$ is
$m(A)=\|A^{-1}\|^{-1}$ the above can be written as
\begin{equation}
\|Df^m_{/E_x}\|\le \frac{1}{2}m(Df^m_{/F_x})
\end{equation}

The above suggest the terminology: the bundle $E$ \textit{dominates}
the bundle $F,$ and sometimes it is written as $E\prec F.$ Condition
\eqref{e.ds} and the equivalent ones also imply that to ``see''
the domination we might have to wait some iterations.

Notice that the Dominated Splitting does not depends on the
riemannian metric, that is, a set having dominated splitting still
it has it no matter if we change the metric on the manifold.
Nevertheless, the constants $C$ and $\la$ above do depend on the
metric. We can always find a riemannian metric on the
manifold so that the constant $C$ above is equal to $1,$ in other
words we can ``see'' the domination in the first step. This is a
result by N. Gourmelon \cite{Go1}.

Let's see now an elementary property:

\begin{prop}\label{p.dominated}
Assume that $\La$ has a Dominated Splitting $T_\La M=E\oplus F.$
Then this Dominated Splitting can be extended to the closure. Also,
condition \eqref{e.cont} is equivalent to the following: the maps
$x\to dim(E_x)$ and $x\to dim(F_x)$ for $x\in\La$ are locally
constant. The angle between the subspaces $E_x$ and $F_x$ is bounded
away from zero.
\end{prop}

\begin{proof}[Sketch of proof]
Let $x_n\to x$ be a sequence in $\La$ such  that
$E_{x_n}$ and $F_{x_n}$ converges to subspaces $\tilde{E}_x$ and
$\tilde{F}_x.$ By defining $\tilde{E}_{f^n(x)}=Df_x(\tilde{E}_x)$
and similarly $\tilde{F}_{f^n(x)}$ we have that $\tilde{E}$ and
$\tilde{F}$ satisfy \eqref{e.ds}. In particular
$T_xM=\tilde{E_x}\oplus \tilde{F}_x.$ Let $y_n$ be also a sequence
in $\La$ such that $y_n\to x$ and that $E_{y_n}$ and $F_{y_n}$
converges to subspaces denoted $E_x$ and $F_x.$ Notice that
$dimE_x=dim\tilde{E}_x$ and $dimF_x=dim\tilde{F}_x.$ We would like
to show that $E_x=\tilde{E}_x.$ Otherwise, let $v\in \tilde{E},
\|v\|=1$ and   $v=v_E+v_F$ with $v_F\neq 0.$ Then
\begin{eqnarray*}
\|Df^n_{/\tilde{E}(x)}\| &\ge& \|Df^nv\|\ge
\|Df^nv_F\|-\|Df^nv_E\|\\
&\ge & m(Df^n_{/F(x)})\|v_F\|-\|Df^n_{E(x)}\|\|v_E\|\\
&=&
\|Df^n_{E(x)}\|\left(\frac{m(Df^n_{/F(x)})}{\|Df^n_{E(x)}\|}\|v_F\|-\|v_E\|\right)\\
&\ge & \|Df^n_{E(x)}\|\left(\frac{\|v_F\|}{C\la^n}-\|v_E\|\right)
\end{eqnarray*}
and so $\frac{\|Df^n_{/\tilde{E}}\|}{\|Df^n_{/E}\|}\to\infty.$
Interchanging the role of $E_x$ and $\tilde{E}_x$  we get a
contradiction. From this one can easily conclude the Proposition.
No matter which reasonable definition of angle between subspaces
we have, it is obvious that it is bounded away from zero.
\end{proof}

A set may have many dominated splittings. Let's say that $\La$ has a
dominated splitting $T_\La M=E\oplus F$ of index $i$ if $dimE_x=i$
for any $x\in\La.$ The above Proposition also shows that a dominated
splitting of index $i$ on a set $\La$ is \textit{unique}.

\begin{defi}\label{d.finest}
Let $\La$ be an invariant set of $f:M\to M.$ Assume that we have a
decomposition $T_\La M=E_1\oplus E_2\oplus\ldots\oplus E_k$
invariant under $Df.$ We say that it is a Dominated Splitting
provided
$$T_\La M=\left(E_1\oplus\ldots\oplus E_j\right)\oplus \left(E_{j+1}\oplus\ldots\oplus E_k\right)$$
is a Dominated Splitting for any $j=1,\ldots, k-1.$ When the extremal subbundles $E_1$ and $E_k$ are uniformly contracting and expanding we call it \textit{partially hyperbolic} (sometimes it is required that just one has a uniform behavior).
\end{defi}

Thus, on an invariant set $\La$ having dominated splitting we can
consider the \textit{finest dominated splitting} as the dominated
splitting $T_\La M=E_1\oplus E_2\oplus\ldots\oplus E_k$ such that no
$E_i$ can be decomposed again so that the whole decomposition is dominated as well. This notion
was introduced in \cite{BDP}. The finest dominated splitting is
unique. See Appendix B of \cite{BDV}.

Now we may ask: given the finest dominated splitting on a set, can
we describe the dynamics? The question is twofold: on one hand what
are the dynamical implications of such dominated splitting? and on
the other one, which are the dynamical phenomena that prevents having
a finer dominated splitting? These are very hard questions and we
will try to shed some light on them in the following sections.

Before we see some examples let's characterize the domination in
terms of cone fields. For $x\in M$ and $a>0$,  an $a$-cone of
dimension $n-i$ is a subset $\cC_a(x)$ of $T_xM$ such that we may
find a direct decomposition $T_xM=\tilde{E}\oplus \tilde{F}$ with
$dim \tilde{E}=i, dim \tilde{F}=n-i$ such that
$$\cC_a=\{v\in T_xM: v=v_{\tilde{E}}+v_{\tilde{F}} \mbox{ such that }\|v_{\tilde{E}}\|\le a\|v_{\tilde{F}}\|\}.$$

\begin{prop}\label{p.cones}
Let $\La$ be an invariant set of $f:M\to M.$ Then $\La$ has a
Dominated Splitting of index $i$ if and only if there exist a map
$a:\La\to \RR^+$ bounded away from zero and infinity, a cone field
$\cC_{a(x)}(x)$ of dimension $n-i$, a number $0<\la<1$ and a
positive integer $n_0$ such that
$$Df^{n_0}_x(\cC_{a(x)})\subset \cC_{\la a(f^{n_0}(x))}.$$
\end{prop}
\begin{proof}[Sketch of Proof]
The direct implication follows immediately. For the converse, lets
assume for simplicity that $n_0=1.$  Define
$$E_x=\bigcap_{n\ge 0}Df^{-n}(\cC_{a(f^n(x)})\,\,\,\mbox{ and }\,\,\,F_x=\bigcap_{n\ge 0}Df^n(\cC_{a(f^{-n}(x))}).$$
It follows that $E_x$ and $F_x$ are subspaces of dimension $i$ and
$n-i,$ and invariant under $Df.$ Now consider a tiny cone
$\cC_\ep(x)$ with respect to $T_xM=E_x\oplus F_x.$ Then there exists
$m>0$ such that $Df^m_x(\cC_{a(x)})\subset\cC_\ep(f^m(x))$ and from
this one gets the domination.
\end{proof}

For a more subtle characterization in terms of cones and a spectral
gap see \cite{BoG} (see \cite{Wj} as well). A characterization of
domination in terms of cones also appeared in \cite{N}.  From the
above Proposition one easily gets the following important property
of dominated splitting: it can be extended to a neighborhood and can not
be destroyed by perturbations.

\begin{prop}
Let $f\in \Di^1(M)$ and let $\La$ be a compact invariant set having
Dominated Splitting. Then there exist a compact neighborhood
$U(\La)$ of $\La$ and a neighborhood $\cU(f)\subset\Di^1(M)$ of $f$
such that for any $g\in\cU(f)$ the compact set
$\displaystyle{\bigcap_{n\in\ZZ}g^n(U)}$ has Dominated Splitting.
\end{prop}

Lets see some simple examples to have in mind:

\begin{itemize}
\item A hyperbolic set is always a set with Dominated Splitting with $E=E^s$ and $F=E^u$ (although it may have other dominated splittings).

\item Let $p$ be a fixed (periodic) point of $f:M\to M.$ Assume $Df_p$ has all eigenvalues of moduli different from  $\sigma>0$ (and having eigenvalues greater and smaller than $\sigma).$ Then $p$ has a Dominated Splitting.

 \item Let $\cC$ be a smooth (at least $C^1$) simple closed curve and invariant by $f:M\to M$ and such that $f_{/\cC}$ is conjugated to an  irrational rotation. Assume that $T_\cC M=E^s\oplus T\cC\oplus E^u$ invariant under $Df$ where $E^s$ is uniformly contracted and $E^u$ is uniformly expanding. This means that $\cC$ is a normally hyperbolic simple closed curve supporting an irrational rotation.

 \item Ma\~n\'e Derived from Anosov diffeomorphism on $\TT^3$ \cite{M4}: it is a diffeomorphism $f:\TT^3\to\TT^3$ such that $T\TT^3=E^s\oplus E^c\oplus E^u$ is a partially hyperbolic splitting. This diffeomorphism have periodic points of different indices, it is transitive and any $C^1$ small perturbation is transitive as well. This example is obtained by modifying a linear Anosov map on $\TT^3$ by forcing a fixed point going trough a picht-fork bifurcation (see also \cite{BV}, \cite{PS1}, \cite{BFSV} for other properties).

\item Bonatti-Viana example on $\TT^4$ \cite{BV}: it is a diffeomorphism $f:\TT^4\to\TT^4$ having a Dominated Splitting $T\TT^4=E^{cs}\oplus E^{cu}$ (which is the finest dominated splitting). It is also obtained by modifying an Anosov on $\TT^4$ with hyperbolic structure $T\TT^4=E^s\oplus E^u$ both bidimensional (see \cite{BV}, \cite{BuFi}, \cite{T1} for other properties as well).

\item Let $f:M\to M$ where $M$ is two dimensional. Let $p$ be a fixed point having an eigenvalue of modulus less than one and the other is one, $T_pM=E^s\oplus F.$ Let $\cW^s(p)$ the (strong) stable manifold of $p$ with  $T_p\cW^s(p)=E^s.$ Assume that there is $C^1$ invariant curve $\cW^{cu}(p)$ containing $p$ and  $T_p\cW^{cu}=F.$ Assume that there exists a point $x$ of transversal intersection between $\cW^s(p)$ and $\cW^{cu}(p)$ and such that $f^{-n}(x)\to_{n\to+\infty} p$ (think of a homoclinic class associated to a saddle node fixed point). Then $\La=\{p\}\cup\{f^n(x): n\in\ZZ\}$ has a dominated splitting.

\end{itemize}

\section{Dominated Splitting and periodic points}

 In this section we present some general results and ideas on Dominated Splitting (related mostly to periodic points). Periodic points play a fundamental role in dynamics from many different viewpoints. One fundamental problem solved by Pugh \cite{Pu1} and known as Pugh's Closing Lemma  states that if $x$ is a nonwandering point of $f:M\to M$ one can perturb $f$ in the $C^1$ topology to obtain $g$ so that $x$ is a periodic point for $g.$ However, from the proof one can not conclude that the $g$-orbit of $x$ shadows the $f$-orbit of $x.$ This difficulty was solved in \cite{M5}.

\begin{thm}[Ma\~n\'e's Ergodic Closing Lemma]\label{t.ergodic}
Let $f\in\Di^1(M)$ and let $\mu$ be an invariant probability measure. Then, for any neighborhood $\cU(f)\subset \Di^1(M)$ and $\ep>0$, for $\mu-a.e. x\in M$ there exists $g\in\cU(f)$ and a $g$-periodic point $p$ such that
$$dist(g^j(p),f^j(x))<\ep\,\,\,\, j=0,1,\dots, n(p)$$
where $n(p)$ is the period of $p.$

Moreover, there exists a residual set $\cD\subset\Di^1(M)$ such that if $f\in\cD$ and $\mu$ is an ergodic probability measure invariant under $f$ there exists a sequence of $f$-periodic points $p_n$ with periods $n(p_n)$ such that the measures $\mu_{p_n}=\frac{1}{n(p_n)}\sum_{j=0}^{n(p_n)-1}\delta_{f^j(p_n)}$ converges to $\mu$ in the weak topology.
\end{thm}

The second part appeared in \cite{M6}. For a surprisingly simple and elegant proof of the above theorem see \cite{C1}. We can see now an important result regarding dominated splitting in \cite{M5}.

\begin{thm}
Let $f\in\Di^1(M)$ and let $\La$ be a compact invariant set having a Dominated Splitting $T_\La M=E\oplus F$ of index $i.$ Then, if this decomposition is not hyperbolic then for any neighborhood $U$  of $\La$ and every neighborhood $\cU\subset\Di^1(M)$ of $f$ there exist $g\in \cU$ and a hyperbolic $g$-periodic point $p$ in $U$ with index different from $i.$
\end{thm}

\begin{proof}[Sketch of proof:]
If the Dominated Splitting is not hyperbolic, then $F$ is not uniformly expanded or $E$ is not uniformly contracted. Assume that $E$ is not uniformly contracted. Then there exists a point $x\in\La$ such that $\Pi_{j=0}^n\|Df_{/E_{f^j(x)}}\|\ge\frac{1}{2}$ for any $n\ge 0.$ Consider the sequence of probability measures $\mu_n=\frac{1}{n}\sum_{j=0}^{n-1}\delta_{f^j(x)}$. Let $\mu$ an accumulation point of these measures. Then $\int \log(\|Df_{/Ex}\|)d\mu\ge 0.$ By the ergodic decomposition theorem, there exits an ergodic measure $\nu$ such that  $\int \log(\|Df_{/Ex}\|)d\nu\ge 0.$ Let $z$ be $\nu$ generic point. Now, we may assume that $U$ and $\cU$ are sufficiently small so that any invariant set of $g\in\cU$ in $U$ has Dominated Splitting of index $i.$ Let $\ep_n\to 0$ . By the Ergodic Closing Lemma there exist $g_n\in\cU$ and $p_n$ a periodic point of $g_n$ such that $dist(f^j(x),g_n^j(p_n))<\ep, j=0,..,n(p_n).$ If $\ep_n$ is small then the $g_n$-orbit of $p_n$ is in $U.$ Then,
\begin{equation}\label{e.periodic}
\prod_{j=0}^{n(p_n)-1}\|Dg_{n/E_{g_n^j(p_n)}}\|\ge 1-\gamma_n
\end{equation} for some  $\gamma_n\to 0$ (that it is chosen in advance). By the domination ones get $$\prod_{j=0}^{n(p_n)-1}\|Dg^{-1}_{n/F_{g_n^j(p_n)}}\|\le \la^{n(p_n)}$$ for some $\la<1.$

Now, \eqref{e.periodic} implies that for large $n$ we have that $Dg_{n/E_{p_n}}^{n(p_n)}$ has an eigenvalue close to one or can be perturbed to have an eigenvalue (close to) one (see \cite{M5}). Using Franks' Lemma (\cite{F}) one can thus obtain $\tilde{g}_n\in\cU$ such that $p_n$ is $\tilde{g}_n$ periodic and the index of $p_n$ is less than $i.$
\end{proof}

The above result says, roughly speaking,  that if a Dominated Splitting of index $i$ over a set $\La$ is not hyperbolic it is due to the presence of hyperbolic periodic points of index $\neq i.$ A major problem is whether these periodic points of different indices are attached to $\La$. Let us be more precise: assume that $\La$ is the homoclinic class of a hyperbolic periodic point $p$ of index $i$, and so, it has a natural continuation (for $g$ close to $f$ consider $\La(g)$ to be the homoclinic   class of the continuation $p_g$ of $p).$ If the Dominated Splitting is not hyperbolic, then we can perturb $f$ so that $\La(g)$ has a periodic point of different index? It is possible to have counter examples of this (see Remark \ref{r.example}) and the right question should be: if the Dominated Splitting is not hyperbolic then we can perturb $f$ so that either $\La(g)$ is hyperbolic or contains a point of different index. Another way to ask the same is assuming some generic conditions on $f:$

\begin{prob}\label{prob.index}
Let $f\in\Di^r(M), (r\ge 1)$ and assume that $f$ is $C^r$ generic. Let $p$ be a hyperbolic periodic point of index $i$ and assume that the homoclinic class $H(p)$ has a Dominated Splitting of index $i.$ If this Dominated Splitting is not hyperbolic, is it true that $H(p)$ contains a periodic point of index $\neq i?$
\end{prob}

This problem has been attacked in various different situations (always in the $C^1$ topology), and with sophisticated techniques  (see for instance \cite{C1}, \cite{C2}, \cite{CP}). A possible approach would  be to show that the homoclinic class $H(p)$ has ``weak
periodic points of index $i$'' (see Problem \ref{prob.weak}). Let's see some results in this direction. Some previous results are needed. The first one is Pliss' Lemma, a fundamental tool to play with dominated splitting\footnote{Another useful tool is Liao's Selection Lemma \cite{L3} but we won't state it since we won't use it here.}.

\begin{thm}[Pliss' Lemma \cite{Pl2}]
Let $f:M\to M$ be a diffeomorphism. Then, given $0<\gamma_1<\gamma_2<1$ there exists a positive integer  $N=N(\gamma_1,\gamma_2)>0$ and $d>0$ such that if for some $x\in M$ and some subspace $E_0\subset T_xM$ and denoting $E_j=Df^jE_0$ we have that $\prod_{j=0}^N\|Df_{/E_j}\|\le \gamma_1^N$ then there exist $0\le n_0<n_1<\ldots<n_k<N$ with $k>dN$ such that
$$\prod_{j=0}^{n-1}\|Df_{/E_{n_i+j}}\|\le\gamma_2^n\,\,\,\,\forall \,\,\,1\le n\le N-n_i,\,\,\,\,i=0,...,k.$$

\end{thm}

The next result is a version of the Anosov Closing Lemma. Let $\La$ be a compact invariant set having Dominated Splitting $T_\La M=E\oplus F.$ Let $0<\eta<1.$ For $x\in\La$ and a positive integer $n$ we say that $(x,f^n(x))$ is $\eta$-string provided
$$\prod_{j=0}^{k-1}\|Df_{/E_{f^j(x)}}\|\le \eta^j,\,k=1 ,...,n \;\mbox{ and }\;
 \prod_{j=0}^{k-1}\|Df^{-1}_{/F_{f^{n-j}(x)}}\|\le \eta^j,\,k=1,...,n$$

\begin{thm}\label{t.closing}
Let $f:M\to M$ be a diffeomorphism and let $\La$ be a compact invariant set having Dominated Splitting $T_\La M=E\oplus F$ and let $0<\eta<1$ be given. Then, given $\ep>0$ there exists $\delta>0$ such that if $(x_i,f^{n_i})$ are $\eta$-strings, $i=1,...,k$ such that $dist(f^{n_i}(x_i),x_{i+1})<\delta,$ for $i=1,...,k-1$ and $dist(f^{n_k}(x_k),x_1)<\delta)$ then there exists a periodic point $p,$ with $f^n(p)=p$ where $n=n_1+...+n_k$ and
$$dist(f^{n_1+...+n_{j-1}+\ell}(p), f^\ell(x_j))<\ep\,\,\,\,1\le j\le k,\,\, 0\le \ell\le n_j.$$
\end{thm}

When $k=1$, an explicit proof was given by Liao \cite{L2}. A proof of the generalized version was done by Gan \cite{Ga1}. The following theorem is a (simplified) version of a result proved by Ma\~n\'e \cite{M1} as an important step for the stability conjecture (Theorem II.1). See also \cite{BGY}. The proof we present here is  different from the one of Ma\~n\'e or \cite{BGY}, although it is based on the same ideas.

\begin{thm}\label{t.manehomoclinic}
Let $p$ be a hyperbolic periodic point. Assume that the homoclinic class $H(p)$ of $p$ has a dominated splitting $T_{H(p)}M=E^s\oplus F$ of the same index as the index of $p$ and where $E^s$ is uniformly contracted. Then, either the homoclinic class is hyperbolic or there are periodic points in the class with arbitrarily weak Lyapunov exponent along $F,$ or more precisely,  given $0<\gamma<1$ there exists a periodic $q\in H(p)$ such that
$$\gamma^{n(q)}<\prod_{j=0}^{n(q)-1}\|Df^{-1}_{/F_{f^{-j}(q)}}\|<1$$
where $n(q)$ is the period of $q.$

\end{thm}

\begin{proof}[Sketch of proof]
Since $p$ is hyperbolic and $index(p)=dim E^s$ we have (perhaps changing the riemannian metric) that $\|Df^{-1}_{/F_{f^j(p)}}\|<\la<1$ for some $0<\la<1.$ (We also may assume that $\|Df_{/E_x}\|<\la$ for any $x\in H(p)$ since $E^s$ is uniformly contracted). From this, it is not difficult to find a dense set $\cD$ in $H(p)$ consisting of periodic orbits such that for any $q\in\cD$ we have $\prod_{j=0}^{n(q)-1}\|Df^{-1}_{/F_{f^{-j}(q)}}\|<\la_1^{n(q)}$ for some $\la<\la_1<1.$ Now, fix any $\gamma$ such that  $\la_1<\gamma<1$ and assume the homoclinic class is not hyperbolic.  Choose $\gamma<\gamma_1<\gamma_2<1.$ From Pliss' Lemma one gets that for every $q\in \cD$ there are points $q_i$ in the orbit of $q$ such that
\begin{equation}\label{e.pliss}
\prod_{j=0}^{n-1}\|Df^{-1}_{/F_{f^{-j}(q_i)}}\|\le \gamma_2^n \mbox{ for any } n\ge 1.
\end{equation}
For each $q\in\cD$ consider all the points $q_1,q_2,...,q_{m_q}=q_1$ that satisfies \eqref{e.pliss} and ordered following the orbit of $q.$ Notice that for any $q_i, q_j$ both $(q_i,q_j)$ and $(q_j,q_i)$ are $\gamma_2$-strings. Now lets look at the orbit distance between $q_i$ and $q_{i+1}$, that is $m_i,$ where $f^{m_i}(q_i)=q_{i+1}.$ If $m_i$ is uniformly bounded for any $i$ and for every $q\in\cD$ we conclude that $H(p)$ is uniformly hyperbolic. Thus, we may find a sequence $q(n)$ in $\cD$ and points $q_{i_n}(n), q_{i_n+1}(n)$ (in the orbit of $q(n)$) whose orbit distance goes to infinity. And so, for $n$ large enough, we get from Pliss Lemma that $\prod_{j=0}^{m_i-1}\|Df^{-1}_{/F_{f^{-j}(q_{i_n+1}(n))}}\|> \gamma_1^{m_i}$ (otherwise we have a point satisfying \eqref{e.pliss} between $q_{i_n}(n)$ and $q_{i_n+1}(n)).$

Now we may assume that $q_{i_n}(n)$ and $q_{i_n+1}(n)$ converge to points $x$ and $y$. Let $\ep>0$ be very small and let $\delta$ form the above closing lemma. Fix some $n_0$ such that $q_{i_{n_0}}(n_0)$ and $q_{i_{n_0+1}}(n_0)$ are at distance less than $\delta/2$ of $x$ and $y.$ And then, choose $n_1$ much larger than $n_0$ so that $q_{i_{n_1}}(n_1)$ and $q_{i_{n_1+1}}(n_1)$ are at distance less than $\delta/2$ of $x$ and $y.$  Notice that the $\gamma_2$ strings $(q_{i_{n_0+1}}(n_0),q_{i_{n_0}}(n_0))$ and $(q_{i_{n_1}}(n_1),q_{i_{n_1+1}}(n_1))$ are as in the Theorem \ref{t.closing}  with $k=2.$ Choose $0<c<1$ such that $\gamma<(1-c)\gamma_1<(1+c)\gamma_2<1$ and set $N=m_{i_{n_0}}+m_{i_{n_1}}$ and let $z$ be the periodic point that shadows these two strings by Theorem\ref{t.closing}, $f^N(z)=z$.

Let $C=\sup_{x\in H(p)}\|Df^{-1}_{/F_x}\|^{-1}.$ If $n_1$ was chosen large enough then we have $$C^{m_{i_{n_0}}}\left((1-c)\gamma_1\right)^{m_{i_{n_1}}}>\gamma^N.$$ Thus, if $\epsilon$ was small enough we conclude that
$\gamma^{N}<\prod_{j=0}^{N-1}\|Df^{-1}_{/F_{f^{-j}(z)}}\|.$ On the other hand $\prod_{j=0}^{n-1}\|Df^{-1}_{/F_{f^{-j}(z)}}\|\le ((1+c)\gamma_2)^n$ for $1\le n\le N.$ This implies that $z$ has a uniform unstable manifold and therefore, and if $\ep$ and $\delta$ were chosen  sufficiently small, we have a heteroclinic intersection between $z$ and $q_{i_{n_0+1}}(n_0)$ and so $z$ belongs to the homoclinic class $H(p).$ . Since $\gamma$ was arbitrary we deduce the existence of periodic points with arbitrarily weak Lyapunov exponent along $F$.

\end{proof}

\begin{prob}\label{prob.weak}
Let $p$ be a hyperbolic periodic point and let $H(p)$ be its homoclinic class. Assume that $Hp)$ has Dominated Splitting $T_{H(p)}M=E\oplus F$ of index equal to the index of $p.$ Is it true that if  neither $E$ is uniformly contracting nor $F$ is uniformly expanding, there are periodic points in the class having a weak Lyapunov exponent along $E$ and periodic points having weak Lyapunov exponent along $F$? (See Problem I.8 of \cite{CSY})
\end{prob}

A partial answer to the above Problem can be given from the results in \cite{CSY}

\begin{thm}
Let $f\in\Di^1(M)$ be a generic diffeomorphism. Let $p$ be a hyperbolic periodic point and assume that the homoclinic class $H(p)$ has dominated splitting $T_{H(p)}=E^s\oplus E^c\oplus F$ where $E^s$ is uniformly contracted, $E^c$ is one dimensional and not uniformly contracted, and $dim(E^s\oplus E^c)$ is equal to the index of $p.$ Then, for any $\delta >0$ there exists periodic points in $H(p)$ whose Lyapunov exponent along $E^c$ belongs to $(-\delta,0).$
\end{thm}

We can refine Problems \ref{prob.index} and \ref{prob.weak} and ask the following:

\begin{prob}\label{prob.stableindex}
Let $H(p)$ be a homoclinic class of a hyperbolic periodic point having a Dominated Splitting $T_{H(p)}=E^s\oplus E^c_1\oplus...\oplus E^c_k\oplus E^u$ where $E^s$ is uniformly contracted and $E^u$ is uniformly expanded and $E^c_i$ are one dimensional and neither $E^c_1$ is uniformly contracted nor $E^c_k$ is uniformly expanded. Do there exist periodic points $q_1, q_2$ in $H(p)$ such that the index of $q_1$ is $dim E^s$ and the index of $q_2$ is $dim E^s+k$? To avoid some  counterexamples one should assume that $f$ satisfies some $C^r$ generic condition (see the end of Section \ref{s.extremal}).
\end{prob}

We remark that from \cite{CSY} the indices $dimE^s+1$ and $dimE^s+k-1$ are realized (and by \cite{ABCDW} the indices between both are realized as well). The above problem admits a more general formulation (and more difficult?) when the class $H(p)$ admits a Dominated Splitting $T_{H(p)}=E^s\oplus E^c \oplus E^u$ where $E^s$ and $E^u$ are maximal (i.e., $H(p)$ does not admit a dominated splitting $\tilde{E}^s\oplus\tilde{E}^c\oplus\tilde{E}^u$ with $E^s\subset \tilde{E}^s$ and $E^u\subset \tilde{E}^u$). A positive answer to the above problem will represent a fundamental step towards the Palis' conjecture or Bonatti's conjecture on finiteness of chain-recurrent classes far from tangencies (see \cite{B}).

\section{Intermediate one dimensional subbundle}

Let $f:M\to M$ be a diffeomorphism and let $\La$ be a compact invariant set having a dominated splitting $T_\La=E\oplus E^c\oplus F$ where $E^c$ is one dimensional. Can we say something about the dynamics of $\La?$ Even in the case where $E=E^s$ is uniformly contracting and $F=E^u$ is uniformly expanding this is a difficult question.

\begin{prob}\label{prob.finit}
Assume that $f\in\Di^r(M)$ is $C^r$ generic and partially hyperbolic $TM=E^s\oplus E^c \oplus E^u$ with $E^c$ one dimensional (and $f$ is not Anosov). Is it true that the chain recurrent set $\cR(f)$ has finitely many chain-recurrent classes? What if $TM=E^s\oplus E^c_1\oplus\ldots\oplus E^c_k\oplus E^u$ with $E^c_i$ one dimensional?
\end{prob}

Related problems to the above are for instance: Let $f$ be time one map of an Anosov flow, is it true that any  $C^r$ generic perturbation  has finitely many chain-recurrent class? Is any $C^r$ perturbation transitive?

There are two main reasons to consider dominated splitting with one-dimensional subbundle. On one hand this is the situation when we are far from homoclinic tangencies (see \cite{PS2}, \cite{We1}, \cite{C2}). On the other one, although the action of $Df$ along $E^c$ might be neutral, since it one dimensional, one could expect to have a description on the dynamics in the $E^c$ direction. This is the content of a basic tool in the area known as Crovisier's Center Models and developed in \cite{C2} and \cite{C3}. This tool was used to prove the major results on the Palis' conjecture (\cite{C3}, \cite{C2} and \cite{CP}) and the Bonatti's conjecture (\cite{B}) on finiteness of chain-recurrent classes far from tangencies  as well (see \cite{CSY}).

Let us briefly explain this tool (we recommend to look at \cite{C1} for this and many applications as well). By the theory in \cite{HPS} one knows that there are a family of one dimensional manifolds $\cW^c(x)$ for $x\in \La$ (center plaques) such that $T_x\cW^c(x)=E^c_x$ and are locally invariant, that is, $f(\cW^c(x))$ contains a neighborhood of $f(x)$ within $\cW^c(f(x)).$ This allows to lift the dynamics to $T_\La E^c$.

\begin{defi}
A center model $(\hat{\La},\hat{f})$ associated to a compact invariant set $\La$ and $f$  with $T_\La=E\oplus E^c\oplus F$ where $dim E^c=1$ is a compact set $\hat{\La},$ a continuous map $\hat{f}:\hat{\La}\times [0,+\infty)$ and a map $\pi:\hat{\La}\times [0,+\infty)\to M$ such that:
\begin{itemize}
\item $\hat{f}(\hat{\La}\times\{0\})=\hat{\La}\times \{0\}.$
\item $\hat{f}$ is a local homeomorphism of a neighborhood of $\hat{\La}\times\{0\}.$
\item $\hat{f}$ can be written as $\hat{f}(x,t)=(\hat{f}_1(x),\hat{f}_2(x,t)).$
\item $\pi(\hat{\La}\times\{0\})=\La.$\
item $\pi\circ\hat{f}=f\circ \pi.$
\item The maps $t\to \pi(\hat{x},t)$ forms a family of $C^1$ embedding of $[0,+\infty)$ in $M.$
\item The curve $\pi(\hat{x}\times [0,+\infty))$ is tangent to $E^c$ at $\pi(\hat{x},0).$
\end{itemize}
\end{defi}

One can consider two cases: the orientable case where there exists an orientation of $E^c$ invariant by $Df$ and where there is not.

\begin{thm}[\cite{C2},\cite{C3}]\label{t.central}
Let $\La$ be a compact, invariant and chain-transitive set with Dominated Splitting $T_\La=E\oplus E^c\oplus F$ where $dim E^c=1.$ Then there exists a center model $(\hat{\La},\hat{f})$ associated to $(\La, f)$ where $\hat{\La}$ is chain-transitive. In the orientable case, $\pi_{/\hat{\La}\times{0}}$ is a homeomorphism and $\pi(\{\hat{x}\}\times[0,+\infty))$ is compatible with the orientation on $E^c,$ and in the non-orientable case $\hat{\La}$ coincides with $T^1_\La E^c.$ Moreover, one of the following cases holds:
\begin{itemize}
\item[Type(R):] (Chain-recurrent) For every $\ep>0$ there exists a point $x\in\La$ and an arc $\gamma, x\in\gamma\subset\cW^c(x)$ such that the length of $f^n(\gamma)$ is bounded by $\ep$ and one can go from $\gamma$ to $\La$ and viceversa with arbitrarily small chains  contained in a neighborhood of $\La.$
\item[Type(N):] (Neutral) There exist a base of attracting open neighborhoods $U$ of the zero section (i.e., $\hat{f}(\overline{U})\subset U$ and a base of repelling open neighborhoods.
\item[type(H):] (Hyperbolic) There exists a base of attracting (respect. repelling) open neighborhoods of the zero section and $\cW^c$ is contained in the chain stable (respect. unstable) set of $\La.$
\item[Type(P):] (Parabolic) In this case we are in the orientable case and on each ``side'' we have type (N) or (H). Thus, three subcases appear: $P_{SU}, P_{NS}, P_{NU}$ depending on which cases appear on $E^{c,+}$ and $E^{c,-}.$

\end{itemize}
\end{thm}

As we said before, using the central models one can get many interesting result. Just to give a rough idea let's see the following:

\begin{thm}[\cite{C2},\cite{C3}]
Let $f\in\Di^1$ be a $C^1$ generic diffeomorphism and let $\La$ be a compact invariant chain-transitive set with Dominated Splitting $T_\La=E^s\oplus E^c\oplus E^u$ where $E^s$ is uniformly contracted, $E^u$ is uniformly expanded and $dim E^c=1$ and of Type(R). Then, $\La$ is contained in the chain-recurrent class of a periodic point.
\end{thm}

\begin{proof}[Sketch of Proof:]
Every point whose orbit remains in a neighborhood of $\La$ has a uniform (strong) stable $\cW^s$ and unstable $\cW^u$ manifolds (tangent to $E^s$ and $E^u$). Let $\gamma$ be a chain recurrent segment. Then $\cup_{c\in\gamma}\cW^s(x)$  and $\cup_{c\in\gamma}\cW^u(x)$ defines (topological) manifolds of dimension $dimE^s+1$ and $dimE^u+1$ and contained respectively in the chain stable and chain unstable set of $\La.$  Since $f$ is generic (see \cite{BC} and \cite{C4}) one find a periodic orbit $\cO$ whose orbit remains in a neighborhood of $\La$ and there is $p\in\cO$ close to the middle point of $\gamma.$ Thus, the strong stable and unstable manifolds at $p$ intersects the chain unstable and chain stable set of $\La.$
\end{proof}

The following result is a consequence of Theorem 1.2 of \cite{CSY}:

\begin{thm}
Let $f$ be a $C^1$ generic diffeomorphism and let $\La$ be a compact invariant chain transitive set having a Dominated Splitting  $T_{\La}M=E^s\oplus E^c_1\oplus\ldots E^c_k\oplus E^u$ where $E^s$ is uniformly contracted, $E^u$ is uniformly expanded and $dim E^c_i=1, i=1,\ldots, k.$ And neither $E^c_1$ is uniformly contracted nor $E^c_k$ is uniformly expanded. Then one of the following holds:
\begin{itemize}
\item $k=1$ and the Lyapunov exponent along $E^c_1$ of every ergodic measure supported in $\La$ is zero.
\item $\La$ is contained in the chain-recurrent class of a periodic point.
\end{itemize}
\end{thm}

The idea is the following: if all Lyapunov exponents along $E^c_1$ of every ergodic measure are zero then we are in the first case. Otherwise there exists an ergodic measure $\mu$  whose Lyapunov exponent is non-zero and Theorem 1.2 of \cite{CSY} applies and we are in the second case.

Furthermore, it follows from the same theorem that, in the second case, $\La$ is contained in the local homoclinic class of a hyperbolic periodic orbit $\cO$ in a neighborhood $U$ of $\La$ (that is, the closure of the set of transversal intersections of the stable and unstable manifolds of $\cO$ whose orbits are contained in $U$)  and so with the same structure $E^s\oplus E^c_1\oplus\ldots E^c_k\oplus E^u.$

\begin{prob}\label{prob.aperiodic}
Let $f$ be a $C^r$ generic diffeomorphism. Does there exist an aperiodic chain recurrent class $\cC$ (i.e. a chain recurrent class without periodic points) partially hyperbolic $T_\cC M=E^s\oplus E^c\oplus E^u$ with $dimE^c=1$?
\end{prob}

A negative answer to this question will represent a definite step towards the Palis's conjecture or the Bonatti's conjecture (see \cite{B}, \cite{CP}, \cite{CSY}).

\section{Extremal one dimensional subbundle}\label{s.extremal}

In this section we study compact invariant sets $\La$ having dominated splitting $T_\La M=E\oplus F$ where $E$ or $F$ is one dimensional. We call this kind of splitting \textit{codimension one dominated splitting}.

Lets begin with the simplest case: when the ambient manifold $M$ is compact surface (i.e. a bidimensional compact riemannian manifold). If $\La$ is a compact set having dominated splitting $T_\La M=E\oplus F$ then both $E$ and $F$ are (extremal and) one dimensional. How such a set can fail to be hyperbolic? There are two trivial counterexamples: either $\La$ contains a non hyperbolic periodic point or $\La$ contains a periodic simple closed curve normally hyperbolic (and hence attracting or repelling) such that at the period the dynamics has irrational rotation number. The next result says that these are the only obstructions provided the diffeomorphism is at least of class $C^2.$

\begin{thm}[\cite{PS2}]\label{t.ps}
Let $f\in\Di^2(M)$ where $M$ is bidimensional. Let $\La$ be a compact invariant set having a Dominated Splitting $T_\La M=E\oplus F.$ Assume that any periodic point of $f$ in $\La$ is hyperbolic and that $\La$ does not contain a periodic simple closed curve supporting an irrational rotation. Then $\La$ is hyperbolic.
\end{thm}

This result implies that if all periodic points hyperbolic and no normally hyperbolic curve supporting an irrational rotation exists (which is a $C^r$ generic condition for $r\ge 1$)  then a set having dominated splitting for $f\in Diff^r, r\ge 2$ is hyperbolic. And also implies that there exists a residual set $\cD\subset \Di^1(M)$ such that if $f\in \cD$ and $\La$ is a compact invariant set having dominated splitting is hyperbolic. To prove this consider $\{U_n\}$ a countable basis of the topology of $M.$ Consider the family $\{\cV_n\}$ of finite collection of elements $U_n$'s (which is countable). Let $\cA_n$ the interior of the set $$\{f\in \Di^1(M): \mbox{ if } K\subset \cV_n \mbox{ is compact, invariant having DS is hyperbolic}\}$$
where DS stands for Dominated Splitting. Let $\cB_n$ be the complement of the closure of $\cA_n.$ Let $\cD_n=\cA_n\cup\cB_n$ and $\cD=\cap_n\cD_n.$ Let $f\in\cD$ and let $K$ a compact invariant set having Dominated Splitting. Let $V$ be a compact neighborhood of $K$ and let $\cU$ be a neighborhood of $f$ such that for any $g\in\cU$ the maximal invariant set of $g$ in $V$ has Dominated Splitting. Now, there exists $\cV_n$ such that $K\subset\cV_n\subset V.$ Taking a sequence $g_n$ of $C^2$ generic diffeomorphism converging to $f$ we see that the maximal invariant set of $g_n$ in $V$ is hyperbolic and so $f$ can not be in $\cB_n.$ Therefore, $f\in\cA_n$ and so $K$ is hyperbolic.

So far there is no proof of this fact just using $C^1$ techniques (i.e. without approximating by a $C^2$ diffeomorphism and using  Theorem \ref{t.ps}).

The proof of Theorem \ref{t.ps} is very technical. It is an extension of a result on non-critical one dimensional dynamics of Ma\~n\'e (\cite{M7}). A very rough and general idea about the proof is: since both $E$ and $F$ are one dimensional extremal bundles, there exist locally invariant manifolds $\cW^{cs}_{loc}(x)$ and $\cW^{cu}_{loc}(x)$ tangent to $E$ and $F$ respectively having dynamics properties: if $y\in \cW^{cs}_{loc}(x)$ then $dist(f^n(y),f^n(x))\to_{n\to\infty} 0$ and similarly for $\cW^{cu}_{loc}(x)$ in the past (in terms of Crovisier's Center Models, they are respectively type(H)-attractive and type(H)-repelling). Then one prove that $\sum_{n\ge 0}|f^n(\cW^{cs}_{loc}(x)|<\infty$ where $|f^n(\cW^{cs}_{loc}(x)|$ denotes the length (the same for $\cW^{cu}_{loc}(x)$ in the past). Finally, the above implies that $\|Df^n_{/E_x}\|\to_{n\to\infty} 0$ and $\|Df^{-n}_{/F_x}\|\to_{n\to\infty} 0.$ For all these facts, the $C^2$ assumption is crucial (in order to control distortion). Indeed, if $f$ is just $C^1$ the above theorem is false: at the end of this section we give a counterexample.

Theorem \ref{t.ps} gives that dominated splitting on two dimensions of a $C^2$ diffeomorphism imposes  certain constraints. Thus, we may ask if we can fully describe the dynamics. The answer is \textit{yes} (at least to some extent).

\begin{thm}[\cite{PS3}]\label{t.dd}
Let $f\in
Diff^2(M^2)$ and assume that the Limit Set $L(f)$ has a dominated splitting. Then
$L(f)$ can be decomposed into
$L(f)=\mathcal{I}\cup{\Li(f)}\cup{\mathcal{R}}$ such that
\begin{enumerate}
\item $\mathcal{I}$ is a set of periodic points with bounded periods
and contained in a disjoint union of finitely many normally
hyperbolic periodic
 arcs or simple closed curves.
\item $\mathcal{R}$ is a finite union of normally hyperbolic periodic simple closed curves
supporting an irrational rotation.
\item $\Li(f)$ can be decomposed into a disjoint union of finitely
many compact invariant and transitive sets. The periodic points are
dense in $\Li(f)$ and contains at most finitely many non-hyperbolic
periodic points. The (basic) sets above are the union of finitely
many (nontrivial) homoclinic classes. Furthermore $f{/\Li(f)}$ is
expansive.
\end{enumerate}
\end{thm}

A fundamental step to prove the theorem above is the following rather surprising result:

\begin{thm}[\cite{PS3}]
Let $f:M\to M$ be a $C^2$-diffeomorphism
of a two dimensional compact riemannian manifold $M$ and let $\La$
be a compact invariant set having dominated splitting. Then, there
exists an integer $N_1>0$ such that any periodic point $p\in\La$
whose period is greater than $N_1,$ is a hyperbolic periodic point
of saddle type.
\end{thm}

Now, let turn to $dim M\ge 3$ and let $\La$ a compact invariant set having codimension one dominated splitting, say $T_\La M=E\oplus F$ where $F$ is one-dimensional. A natural extension of Theorem \ref{t.ps}  should be: similar conditions as in Theorem \ref{t.ps} imply that $F$ is uniformly expanding?. Some partial results were given in \cite{PS4} (when $E$ is uniformly contracted) and in \cite{CP} with some condition on the topology of $\La.$ Nevertheless, it turns out to be true:

\begin{thm}[\cite{CPS}]
Let $f:M\to M$ be a $C^2$ diffeomorphism and let $\La$ be a compact invariant set having a dominated splitting $T_\La M=E\oplus F$ with $dim F=1.$ Assume that for any periodic points in $\La$ the Lyapunov exponent  $F$ is positive and that $\La$ does not contain a periodic simple closed curve tangent to $F$ and supporting an irrational rotation. Then $F$ is uniformly expanded.
\end{thm}

We finish this section giving a $C^1$ counterexample for Theorem \ref{t.ps}:

\begin{thm}\label{t.example}
There exists a $C^1$ diffeomorphism $g:\TT^2\to\TT^2$ having a Dominated Splitting $T\TT^2=E\oplus E^u$  such that $E^u$ is uniformly expanded and:
\begin{itemize}
\item Any periodic point of $g$ is hyperbolic of saddle type
\item $g$ is in the $C^1$ closure of the set of Anosov diffeomorphisms.
\item $g$ is conjugated to an Anosov diffeomorphism.
\item $E$ is not uniformly contracted.
\end{itemize}
\end{thm}

For the (sketch of the) proof of this theorem we need some auxiliary lemmas. The first one is straightforward.

\begin{lem}\label{l.dom} Given, $\sigma>1$ there exist $\eta>0, \beta>0$ and $\sigma_1>1$ such that
given  $h$  a local diffeomorphism around $0$ in $\RR^2, h(0)=0$ and assume that $h$ can be written as $h(x,y)=(u(x,y),v(y))$ and such that $\|u_x\|\le e^\beta, \|u_y\|<\eta$ and $\|v_y\|>\sigma$ then:
\begin{itemize}
\item Denoting $\cC_a^u=\{(u,v)\in\RR^2: \|u\|\le a\|v\|\}$ we have $Dh_{(x,y)}\cC^u_1\subset \cC^u_\rho$ where $\rho=\frac{e^\beta+\eta}{\sigma}<1.$
\item If $w\in\cC_1^u$ then $\|Dh_{(x,y)}w\|>\sigma_1.$
\end{itemize}

\end{lem}

The next lemma will be the key to the induction argument we are going to do:

\begin{lem}\label{l.pert}
Let $h(x,y)=(u(x,y),v(y))$ be as the above Lemma. Let $\delta>0, \alpha>0, \gamma>0, \beta>0$ and $\la>-\beta$ be given with $e^\beta<2$ and assume that $0<u_x\le e^{-\la}$ then there exists a local diffeomorphism $g$ which coincides with $h$ outside a ball of radius $\delta$ at the origin and such that $g$ can be written as $g(x,y)=(\tilde{u}(x,y),v(y))$ and such that:
\begin{itemize}
\item $\|\tilde{u}_x\|\le e^{\gamma-\la}$ and $\|\tilde{u}_y\|<\eta.$
\item $dist_{C^0}(h,g)<\alpha.$
\item $dist_{C^1}(h,g)< 2(e^\gamma-1).$
\item $Dg_{(0,0)}(1,0)=(\tilde{u}_x(0,0),0)=e^{\gamma}(u_x(0,0),0)=e^{\gamma}Dh_{0,0}(1,0).$
\end{itemize}
\end{lem}

\begin{proof}
Let $Z:\R\rightarrow\R$ be a $C^{\infty}$ bump function such
that $Z(0)=1$, $supp[Z]\subset (-\frac{\delta_1}{2},\frac{\delta_1}{2})$ (where $supp[Z]$ is the support of $Z$) and
$|Z'(t)|<\frac{4}{\delta_1}$ where $\delta_1=\delta/2.$

We need the following auxiliary lemma (See Lemma 2.0.1 of \cite{PaS})

\noindent\textbf{Sublemma:}\textit{
For all $k>0$ arbitrarily small and $\gamma>0$ there exist a function
$\beta_k:[0,+\infty)\rightarrow\R$ such that:
\begin{enumerate}
\item $\beta_k$ is $C^{\infty}$, non-increasing and such that $-k
\leq \beta_k'(t)t \leq 0$.
\item $\beta_k$ is supported in $[0.k]$, i.e. $supp[\beta_k]\subset [0,k].$
\item $\beta(0)=e^{\gamma}-1>0.$
\end{enumerate}}

Define $$g(x,y)=h(x,y)+(Z(y)\beta(x^2)u_x(0,0)x,0).$$
Notice that $\tilde{u}(x,y)=u(x,y)+Z(y)\beta(x^2)u_x(0,0)x.$
Let's see that the conditions are fulfilled if $k$ in the above sublemma was chosen small enough. It is immediately that $g$ and $h$ coincides outside a ball of radius $\delta$ at the origin.

Notice that $\|g(x,y)-h(x,y)\|\le \beta(0)e^\beta k<\alpha$ if $k$ is small enough. Now, $Dg=Dh+ A$ with $A=\left(\begin{array}{cc}
a_1 & a_2\\0 & 0
\end{array}\right)$ and
$$a_1=Z(y)[\beta(x^2)u_x(0,0)+2\beta'(x^2)x^2u_x(0,0)]\,\,\,\,\mbox{ and }\,\,\,\, a_2=Z'(y)\beta(x^2)u_x(0,0)x.$$

Then, $\|a_1\|\le \beta(0)e^{-\la}+2ke^{-\la}<2\beta(0)$ if $k$ is small. Observe that $\beta(0)=e^{\gamma}-1.$
On the other hand $\|a_2\|\le \frac{4}{\delta_1}\beta(0)e^{-\la} k<\beta(0)$ again if $k$ is small enough.

Moreover $\tilde{u}_x(0,0)=u_x(0,0)+\beta(0)u_x(0,0)=e^{\gamma}u_x(0,0).$ Finally, since $u_x>0$ and $-k\le\beta'(t)t\le 0$ we have that $\tilde{u}_x>0$ if $k$ is small and
\begin{eqnarray*}
\|\tilde{u}_x\|& \le &\|u_x\|+\beta(0)\|u_x(0,0)\| \le  e^{-\la}(1+\beta(0))=e^{\gamma-\la}
\end{eqnarray*}
\end{proof}

Let's continue with the proof of Theorem \ref{t.example}. The idea is to begin with a linear  Anosov and then perform a sequence of perturbations (along periodic points) that converge in the $C^1$ topology to the desired diffeomorphism.

Start with the linear Anosov diffeomorphism $f_0:\TT^2\to\TT^2$ given by the matrix $A=\left(\begin{array}{cc} 2&1\\1&1
\end{array}\right).$ Let $e^{-\la}$ and $e^\mu$ be the eigenvalues of $A$ where $\la$ and $\mu$ are positive and let $E^s, E^u$ be the associated subspaces. Everything we do in local coordinates will be referred to the decomposition $E^s\oplus E^u$ (both in the tangent space as in $\TT^2$). Choose $1<\sigma<e^\mu$ and then set $\beta>0, \eta>0$ and $\sigma_1$ from the first lemma (and $\rho$ as well).

Choose a sequence $\ep_n, n\ge 1$ such that $\sum_{n\ge 1}\ep_n<\beta.$ Now, set a sequence of positive numbers $\la_n, n\ge 0$ as follows: $\la_0=\la$ and for $n\ge 1:$
$$\frac{\la_{n-1}}{2}<\la_n<\frac{\la_{n-1}}{2}+\ep_n.$$ Let $\tilde{\la}_n,n\ge 0$ be a sequence such that
$$\tilde{\la_0}=\la_0\,\,\,\,\,\mbox{ and }\,\,\,\,\,\frac{\la_{n-1}}{2}<\la_n<\tilde{\la}_n<\frac{\la_{n-1}}{2}+\ep_n\,\,\,\mbox{ for }n\ge 1.$$
Set $\gamma_n=\tilde{\la}_n-\la_{n+1}, n\ge 0$  and observe that $\sum_{n\ge 0}\gamma_n<\la_0+\sum\ep_n<\infty$ and that $\sum(\tilde{\la_n}-\la_n)<\sum\ep_n<\beta.$

Now, let's begin our induction process. Let $p=p_0$ the fixed point of $f_0.$ Applying Lemma \ref{l.pert} around the fixed point $p_0$ with $\la=\la_0$ and $\gamma=\gamma_0$ ($\delta$ and $\alpha$ do not matter too much here), we get a diffeomorphism $f_1:\TT^2\to\TT^2$ where  the stable foliation (tangent to $E^s$)  is kept invariant, $p_0$ is fixed by $f_1$, and \begin{itemize}
\item The Lyapunov exponent of $f_1$ at $p_0$ is $L^s(p_0,f_1)=-\la_0+\gamma_0=-\la_1.$ That is $\|Df_{1/E^s}\|=e^{-\la_1}.$
\item $dist_{C^1}(f_0,f_1)\le 2\left(e^{\gamma_0}-1\right).$
\item $\|Df_{1/E^s_x}\|\le e^{\gamma_0-\lambda_0}\le e^{\la_1}<1<e^\beta.$ for all $x\in\TT^2.$
\item $f_0$ and $f_1$ coincides outside a $\delta$-neighborhood of $p_0.$
\item $f_1$ has a dominated splitting (by Lemma \ref{l.dom}) with expanding direction in the cone $\cC^u_{\rho}$
\item $f_1$ is Anosov.
\end{itemize}
The true induction process starts here: pick a periodic point $p_1$ of $f_1$ with large period $n_1$ that spends much of its time near $p_0$ so that its (stable) Lyapunov exponent is $L^s(p_1)=\frac{1}{n_1}\log\|Df_{1/E^s_{p_1}}\|=-\tilde{\la}_1$ for some $\tilde{\la}_1$ where $\la_1<\tilde{\la}_1<\la_0/2+\ep_1.$

Now, pick $\delta>0$ such that $f_1^j(B_\delta(f^i(p_1))), j=0,...,n_1-1$ are disjoint for any $i=0,...,n-1$ and that $p_0$ does not belong to $B_\delta(\cO(p_1))=\cup_{i=0}^{n_1-1} B_\delta(f^i(p_1)).$ Let $\alpha_1>0$ be such that if $dist_{C^0}(f_1,g)\le 2\alpha_1$ then $g^j(B_\delta(f^i(p_1))), j=0,...,n_1-1$ are disjoint for any $i=0,...,n-1$.

Putting local coordinates at $f^i(p_1)$ and $f^{i+1}(p_1)$ for $i=0,...,n_1-1$ we perform a perturbation using Lemma \ref{l.pert} with $\la=\la_1, \gamma=\gamma_1$  and we get $f_2:\TT^2\to\TT^2$ such that
\begin{itemize}
\item $f_2^j(p_1)=f_1^j(p_1), j\ge 0.$
\item The stable foliation (tangent to $E^s$) is kept invariant.
\item $L^s(p_1,f_2)=\gamma_1+L^s(p_1,f_1)=\gamma_1-\tilde{\la}_1=-\la_2.$
\item $\|Df_{2/E^s_x}\|\le e^{\gamma_1+\gamma_0-\la_0}<e^{\sum \ep_n}<e^\beta.$
\item $dist_{C^0}(f_1,f_2)<\alpha_1.$
\item $dist_{C^1}(f_1,f_2)\le 2\left(e^{\gamma_1}-1\right).$
\item $f_2$ has dominated splitting (by Lemmas \ref{l.dom} and \ref{l.pert}) with expanding direction in the cone $\cC^u_{\rho}$
\item $f_2$ is Anosov.
\end{itemize}

The last point deserves a bit explanation (since at some point in the process the inequality $\|Df_{{n+1}/E^s_x}\|\le e^{\gamma_n+...+\gamma_1+\gamma_0-\la_0}$ does not ensure that $\|Df_{{n+1}/E^s_x}\|<1$ although is less than $e^\beta$ enough to guarantee the domination). The general argument is as follows: since $f_1$ is Anosov, we know that for some $m>0$ we have $\|Df^m_{1/E^s_x}\|<1$ for any $x;$ on the other hand $p_1$ is a periodic point of period $n_1$ with $L^s(p_1,f_1)=-\la_1$ and so, for some neighborhood $U$ of the orbit of $p_1$ we have that $\|Df^{n_1}_{1/E^s_x}\|\le e^{(-\tilde{\la}_1+\ep)n_1}$ where $\ep$ is such that $-\la_2+\ep<0:$ finally perform the perturbation in a tiny neighborhood of the orbit of $p_1$ such that any point outside $U$ is outside the support of the perturbation in the next $m$ iterates. The above implies that $f_2$ is Anosov, since if $x\in U$ then $\|Df^{n_1}_{2/E^s_x}\|\le e^{(-\la_2+\ep)n_1}$ and if $x\notin U$ the $\|Df^m_{2/E^s_x}\|=\|Df^m_{1/E^s_x}\|<1.$

One last remark on how is chosen the $\alpha_n=dist_{C^0}(f_n,f_{n+1})$ (where $\alpha_n$ plays the role for $f_n$ as $\alpha_1$ with $f_1$): $\sum_{n> m}\alpha_n\le \alpha_m.$

Therefore, inductively we have a sequence $f_n$ of Anosov diffeomorphism on $\TT^2$ such that $dist_{C^1}(f_n,f_{n+1})\le 2\left(e^{\gamma_n}-1\right),\, dist_{C^0}(f_n,f_{n+1})\le \alpha_n$, and $f_n$ has a periodic point $p_n$ whose stable Lyapunov exponent is $-\la_n$. The sequence $f_n$ is also uniformly dominated (i.e same cones and estimates) and the foliation tangent to $E^s_A$ is kept invariant. Since $\sum\gamma_n<\infty$ we have that $\{f_n\}$ is a $C^1$-Cauchy sequence and thus converges to a $C^1$ diffeomorphism $g$, having a dominated splitting on $\TT^2.$  The diffeomorphism $g$ is not Anosov since $g$ has a sequence $p_n$ of periodic points whose (stable) lyapunov exponent converges to zero. Every periodic point of $g$ is hyperbolic: let $q$ be periodic of period $k,$ then for $n\ge k$ we have that $q$ is disjoint of the support of perturbation of $f_n$ (since $d_{C^0}(g,f_n)\le\alpha_n)$ and so $g=f_k$ along the orbit of $q$ and so it is hyperbolic. It is not difficult to see that $g$ is expansive and so it is conjugated to Anosov \cite{Le} (or see that the lift of $g$ to $\RR^2$ has infinity as  expansivity constant). This concludes the proof of Theorem \ref{t.example}.

\begin{rem}\label{r.example}
The above example does not satisfy certain $C^1$ generic conditions. For instance, it has no periodic attractor but by an arbitrarily small perturbation one can create them.

Notice also that by multiplying this example with a strong contraction we have a diffeomorphism on a manifold having a homoclinic class with dominated splitting which is not hyperbolic and by perturbation we can not have a periodic point in the class of different index! (recall Problem \ref{prob.index}). However, the homoclinic class can be perturbed to be hyperbolic
 Besides, it also gives an example of a homoclinic class $H(p)$ with dominated splitting $T_{H(p)}M=E^s\oplus E^c\oplus E^u$ where $E^c$ is not contracted but no periodic point in the class has index $dim E^s$ (see Problem \ref{prob.stableindex}). Taking into the account the results of this section, one may ask: \textit{does smoothness has a role to play in Problems \ref{prob.stableindex} and \ref{prob.aperiodic}?}

\end{rem}

\end{document}